\newif\ifarxiv
\newif\ifapt

\arxivtrue\aptfalse

\ifarxiv
\documentclass[a4paper, 11pt]{article}
\usepackage[utf8]{inputenc}
\usepackage{amsmath, amsthm, amssymb}
\usepackage{xspace}
\usepackage{graphicx}
\usepackage{enumerate}
\usepackage[colorlinks=true, allcolors=blue]{hyperref}
\fi

\ifapt
\documentclass{aptpub}
\usepackage{xspace}
\usepackage{enumerate}
\usepackage{color}
\usepackage{graphicx}
\usepackage[colorlinks=true, allcolors=blue]{hyperref}
\fi

\numberwithin{equation}{section}

\ifarxiv
\numberwithin{equation}{section}
\theoremstyle{plain}
\newtheorem{theorem}{Theorem}[section]
\newtheorem{lemma}[theorem]{Lemma}
\newtheorem{corollary}[theorem]{Corollary}
\newtheorem{proposition}[theorem]{Proposition}

\theoremstyle{definition}

\newtheorem{remark}[theorem]{Remark}
\fi

\newcommand{\R}{\mathbb{R}}
\newcommand{\cA}{\mathcal{A}}
\newcommand{\cF}{\mathcal{F}}
\newcommand{\cG}{\mathcal{G}}

\newcommand{\cR}{\mathcal{R}}

\ifarxiv
\newcommand{\E}{\mathbb{E}}
\fi
\ifapt
\renewcommand{\E}{\mathbb{E}}
\fi

\ifarxiv

\fi

\newcommand{\pr}{\mathbb{P}}

\newcommand{\Var}{\operatorname{Var}}
\newcommand{\prto}{\xrightarrow{\rm p}}
\newcommand{\wprto}{\ \xrightarrow{\rm p} \ }
\newcommand{\whp}{\text{w.h.p.}}

\newcommand{\eqd}{\stackrel{\rm d}{=}}
\newcommand{\weq}{\ = \ }

\newcommand{\wle}{\ \le \ }
\newcommand{\wge}{\ \ge \ }
\newcommand{\weqd}{\ \eqd \ }

\newcommand{\abs}[1]{\lvert #1 \rvert}

\newcommand{\Erdos}{Erd\H{o}s\xspace}

\newcommand{\Leskela}{Leskel{\"a}\xspace}

\newcommand{\flood}{\operatorname{flood}}

\newcommand{\cred}[1]{#1}

\begin{document}

\title{First passage percolation on sparse random graphs with boundary weights}

\ifarxiv
\author{Lasse \Leskela \and Hoa Ngo}
\maketitle
\fi

\begin{abstract}
A large and sparse random graph with independent exponentially distributed link weights can be used to model the propagation of messages or diseases in a network with an unknown connectivity structure. In this article we study an extended setting where also the nodes of the graph are equipped with nonnegative random weights which are used to model the effect of boundary delays across paths in the network. Our main results provide approximative formulas for typical first passage times, typical flooding times, and maximum flooding times in the extended setting, over a time scale logarithmic with respect to the network size.
\end{abstract}

\ifarxiv
\fi

\section{Introduction}

Classical first passage percolation theory, initiated about a half century ago in \cite{Hammersley_Welsh_1965}, studies a connected undirected graph $G$ where each adjacent node pair $e$ is attached a weight $W(e) > 0$. When the weights are independent and identically distributed random variables, then
\[
 W_G(u,v)
 \weq \inf_{\Gamma} \, \sum_{e \in \Gamma} W(e),
\]
where the infimum is taken over all paths $\Gamma$ in graph $G$ from $u$ to $v$, defines a natural random metric which has been intensively studied in a wide variety of settings, especially integer lattices \cite{Auffinger_Damron_Hanson_2015}. The quantity $W_G(u,v)$ may be interpreted as the \emph{first passage time} from $u$ to $v$, when the link weights are considered as transmission times. A relevant quantity of interest in modern social and information networks is the \emph{flooding time} $\max_v W_G(u, v)$, which corresponds to the time it takes for a message or disease to spread from a single root node $u$ to all other nodes along the paths of the graph. Alternatively, the link weights can be viewed as economic costs, congestion delays, or carrying capabilities that can be encountered in various real networks \cite{Newman_2010,VanMieghem_2014}.
 
In this paper we study a generalized version of the above setting where in addition to link weights, each node is assigned two weights $X_0(v) \ge 0$ and $X_1(v) \ge 0$, and we define
\[
 W(u,v)
 \weq X_0(u) + W_G(u,v) + X_1(v).
\]
When the weights are considered as transmission times, $W(u,v)$ can be interpreted as the first passage time from $u$ to $v$ in a setting where $X_0(u)$ represents the entry delay and $X_1(v)$ the exit delay along a path from $u$ to $v$ in a network modeled by the graph $G$. The above formulation can also corresponds a generalization of the SI epidemic model \cite{Andersson_Britton_2000} with incubation times by setting $X_0(v) = 0$ and letting $X_1(v)$ represent the length of the time period during which an infected individual $v$ spreads a disease while displaying no symptoms of illness. In this case $W_G(u,v)$ represents the time until node $v$ becomes infected, and $W(u,v)$ the time until node $v$ becomes acutely ill in a population where initially node $u$ is ill and all other nodes are susceptible.

The main result\cred{s} of the paper \cred{are} approximative formulas for $W(u,v)$, $\max_v W(u,v)$, and $\max_{u,v} W(u,v)$ in a large and sparse random graph $G$, when the link weights $(W(e))_{e \in E(G)}$ and the node weights $(X_0(v), X_1(v))_{v \in V(G)}$ are mutually independent collections of independent random numbers, such that $W(e)$ is exponentially distributed with rate parameter $\lambda > 0$, and the distribution of $X_i(v)$ has an exponential tail with rate parameter $\lambda_i \in (0,\infty]$ in the sense that
\begin{equation}
 \label{eq:ExpTail}
 \lim_{t \to \infty} \frac{-\log \pr(X_i(v) > t)}{t} 
 \weq \lambda_i,
 \qquad i=0,1.
\end{equation}
The case $\lambda_i = \infty$ includes distributions with bounded support, for example the uniform distribution on $[0,1]$, and the degenerate case with $X_i(v)=0$ almost surely.  No restrictions about the joint distribution of $X_0(v)$ and $X_1(v)$ are required for the main results.

\paragraph{Notations.} A large network is modeled as a sequence of graphs indexed by a scale parameter $n=1,2,\dots$ Hence most scalars, probability distributions, and random variables depend on $n$, but this dependence is often omitted for clarity. Especially, we write $\pr$ instead of $\pr_n$ for the probability measure characterizing events related the model with scale parameter $n$. An event depending on $n$ is said to occur \emph{with high probability} if its probability tends to one as $n \to \infty$. The symbol $\prto$ refers to convergence in probability. We write $f(n) = o(g(n))$ if $\lim_{n \to \infty} f(n)/g(n) = 0$, and $f(n) = O(g(n))$ if $\limsup_{n \to \infty} f(n)/g(n) < \infty$. We write $X \eqd Y$ when random variables $X$ and $Y$ have the same distribution. \cred{The positive part of a number $x$ is denoted $(x)_{+} = \max\{x,0\}$.}

\section{Main results}
\label{sec:MainResults}
Given a list of nonnegative integers $d = (d_1, \dots, d_n)$, let $G = G(n,d)$ be a random graph, which is uniformly distributed in the set $\cG(n,d)$ of all undirected graphs on node set $[n] = \{1,\dots,n\}$ such that node $v$ has degree $d_v$ for all $v$. We assume that the degree list $d$ satisfies the \Erdos--Gallai condition \cite[Theorem C.7]{Marshall_Olkin_Arnold_2011}, so that $\cG(n,d)$ is nonempty. A stochastic model for a sparse large graph is obtained by considering a sequence of random graphs $G = \cG(n,d^{(n)})$ with degree lists $d^{(n)} = (d_1^{(n)}, \dots, d^{(n)}_n)$ indexed by $n=1,2,\dots$ such that the empirical degree distribution
\[
 f_n(k)
 \weq \frac{1}{n} \sum_{v=1}^n 1(d^{(n)}_v = k)
\]
converges to a limiting probability distribution $f$ with a nonzero finite mean $\mu = \sum_k k f(k)$ according to
\begin{equation}
 \label{eq:DegreeLim}
 f_n(k) \ \to \ f(k)
 \quad \text{for all $k \ge 0$}.
\end{equation}
Throughout we will also assume that for all $n$,
\begin{equation}
 \label{eq:DegreeVariance}
 \sum_k k^{2+\epsilon} f_n(k)
 \wle c
\end{equation}
and
\begin{equation}
 \label{eq:DegreeMinMax}
 \min_v d^{(n)}_v
 \wge \delta
\end{equation}
for some constants $c, \epsilon > 0$ and $\delta \ge 3$ such that $f(\delta) > 0$. Condition~\eqref{eq:DegreeVariance} implies that the family of probability measures $(f_n)_{n \ge 1}$ is relatively compact in the 2-Wasserstein topology \cite{Leskela_Vihola_2013} and guarantees that the mean and the variance of the empirical degree distribution converge to finite values which are equal to the mean and variance of the limiting distribution. Condition~\eqref{eq:DegreeMinMax} in turn implies that $G$ is connected with high probability \cite{Amini_Draief_Lelarge_2013,VanDerHofstad_2018}. 

The following theorem summarizes the main results of the paper. Here $u^*$ and $v^*$ represent uniformly and independently randomly chosen nodes, corresponding to typical values of the quantities of interest.

\begin{theorem}
\label{the:DelayedFloodingTime}
Let $G = G(n,d^{(n)})$ be a random graph satisfying the regularity conditions \eqref{eq:DegreeLim}--\eqref{eq:DegreeMinMax}. Then for independent and uniformly random nodes $u^*$ and $v^*$,
\begin{align}
 \label{eq:DelayedPassageTypical}
 \frac{W(u^*,v^*)}{\log n}
 &\wprto \frac{1}{\lambda(\nu-1)}, \\
 \label{eq:DelayedFloodingTypical}
 \frac{\max_v W(u^*,v)}{\log n}
 &\wprto \frac{1}{\lambda(\nu-1)} + \frac{1}{\lambda \delta \wedge \lambda_1}, \\
 \label{eq:DelayedFloodingMax}
 \frac{\max_{u,v} W(u,v)}{\log n}
 &\wprto \frac{1}{\lambda \delta \wedge \lambda_0} + \frac{1}{\lambda(\nu-1)} + \frac{1}{\lambda \delta \wedge \lambda_1}, 
\end{align}
where $\nu = \sum_k k(k-1) f(k) / \sum_k k f(k)$.
\end{theorem}

\section{Discussion and applications}

\subsection{Earlier work}
The results of Section~\ref{sec:MainResults} are structurally similar to the main result in \cite{Janson_1999} which states that for the complete graph $K=K_n$ on $n$ nodes, the weighted distances (without boundary weights) satisfy
\begin{align}
 \label{eq:UndelayedPassageTypicalK}
 \frac{ W_K(u^*,v^*)}{\log n / n}
 &\wprto \frac{1}{\lambda}, \\
 \label{eq:UndelayedFloodingTypicalK}
 \frac{\max_v W_K(u^*,v)}{\log n / n}
 &\wprto \frac{2}{\lambda}, \\
 \label{eq:UndelayedFloodingMaxK}
 \frac{\max_{u,v} W_K(u,v)}{\log n / n}
 &\wprto \frac{3}{\lambda}.
\end{align}
The above results have more recently been extended to sparse random graphs. For a random graph $G = G(n,d^{(n)})$ satisfying the regularity conditions \eqref{eq:DegreeLim}--\eqref{eq:DegreeMinMax}, the weighted distances (without boundary weights) satisfy
\begin{align}
 \label{eq:UndelayedPassageTypical}
 \frac{W_G(u^*,v^*)}{\log n}
 &\wprto \frac{1}{\lambda(\nu-1)}, \\
 \label{eq:UndelayedFloodingTypical}
 \frac{\max_v W_G(u^*,v)}{\log n}
 &\wprto \frac{1}{\lambda(\nu-1)} + \frac{1}{\lambda \delta}, \\
 \label{eq:UndelayedFloodingMax}
 \frac{\max_{u,v} W_G(u,v)}{\log n}
 &\wprto \frac{1}{\lambda(\nu-1)} + \frac{2}{\lambda \delta }.
\end{align}
Formulas \eqref{eq:UndelayedPassageTypical}--\eqref{eq:UndelayedFloodingMax} agree with \eqref{eq:UndelayedPassageTypicalK}--\eqref{eq:UndelayedFloodingMaxK} because $\nu \approx n$ and $\delta \approx n$ for the complete graph on $n$ nodes. Formula \eqref{eq:UndelayedPassageTypical} was proved in \cite{Ding_Kim_Lubetzky_Peres_2010} for degenerate degree distributions (random regular graph),  in~\cite{Bhamidi_VanDerHofstad_Hooghiemstra_2010} for power-law degree distributions (when $\tau\in (2,3)$), and in \cite{Amini_Draief_Lelarge_2013} for general limiting degree distributions with a finite variance. Formulas \eqref{eq:UndelayedFloodingTypical}--\eqref{eq:UndelayedFloodingMax} have been proved in \cite{Ding_Kim_Lubetzky_Peres_2010} for random regular graphs and in \cite{Amini_Draief_Lelarge_2013,Amini_Lelarge_2015} for general limiting degree distributions with a finite variance. Sparse random graphs where the limiting degree distribution has infinite variance have in general a completely different behavior with typical passage times of order $o(\log n)$ \cite{Baroni_VanDerHofstad_Komjathy_2017,Bhamidi_VanDerHofstad_Hooghiemstra_2010} and they are not discussed further in this paper. The constant $\nu$ appearing in the above formulas can be recognized as the mean of the downshifted size biasing \cite{Leskela_Ngo_2017} of the limiting degree distribution $f$, and $\nu$ is finite if and only if the second moment of $f$ is finite.

Theorem~\ref{the:DelayedFloodingTime} generalizes formulas \eqref{eq:UndelayedPassageTypical}--\eqref{eq:UndelayedFloodingMax} to the setting where nodes have nonnegative random weights $X_0(v)$ and $X_1(v)$ with exponential tail. The main qualitative findings are that the boundary weights have no effect on the typical passage time $W(u^*, v^*)$, but they may affect the typical flooding time $\max_v W(u^*,v)$ and the maximum flooding time $\max_{u,v} W(u,v)$. \cred{All boundary weight effects can be ignored on the $\log n$ time scale when the tails of the node weight distributions decay sufficiently fast $(\lambda_0,\lambda_1 > \lambda\delta)$.}

A notable feature of the results in Theorem~\ref{the:DelayedFloodingTime} is that the leading role of the node weight distributions is the behavior of $\pr(X_i(v) > t)$ as $t \to \infty$, whereas the leading role of link weight distribution is in many cases \cite{Baroni_VanDerHofstad_Komjathy_2017,Janson_1999} governed by the behavior of $\pr( W(e) > t )$ as $t \to 0$.

\cred{
\begin{remark}
The distribution of the node weight $X_i(v)$ is heavy-tailed if the limit in \eqref{eq:ExpTail} is zero. For heavy-tailed node weight distributions, it is easy to check that $\max_{u\in V}X_i(u)$ grows to infinity faster than logarithmically. Hence Theorem \ref{the:DelayedFloodingTime} remains formally valid also when $\lambda_0 = 0$ or $\lambda_1 = 0$, using the convention that $\frac{1}{0} = \infty$.
\end{remark}
}

\subsection{Application: Broadcasting on random regular graphs}

As an application, we discuss a continuous-time version of a message transmission and replication model operating in a \emph{push} mode \cite{Aalto_Leskela_2015,Amini_Draief_Lelarge_2013,Pittel_1987}. Let $G$ be a random $\delta$-regular graph on $n$ nodes, where each node has a state in $\{0,1,2\}$. 
Initially one of the nodes called \emph{root} is in state 1, and all other nodes are in state 0.  Each node activates at random time instants according to a Poisson process of rate $\kappa > 0$, independently of other nodes and the underlying graph structure. When a node activates, it contacts a random target among its neighbors. The states of the nodes are updated in two ways:
\begin{itemize}
\item $0 \mapsto 1$: If the initiator of a contact is in state 1 or 2, and the target node is in state 0, then the state of the target node changes from 0 to 1; otherwise nothing happens during the contact.
\item $1 \mapsto 2$: Having entered state 1, node $v$ remains in this state for a random time period of length $X_1(v)$, and then the state of node $v$ changes into~2.
\end{itemize}
We can interpret the model in the context of computer or biological viruses as follows: State 0 refers to nodes which are vulnerable of receiving a virus. State~1 refers to nodes carrying and spreading the virus but displaying no symptoms. State~2 refers to nodes carrying and spreading the virus and displaying symptoms. We denote by $\flood_1(G)$ the time until every node in the graph has received the virus, and by $\flood_2(G)$ the time until every node displays symptoms.

The above model can be analyzed using the weighted random graph where all links have a random exponentially distributed weight of rate parameter $\lambda = \kappa/\delta$ with $X_0(v) = 0$, and $X_1(v)$ modeling the delay until an infected node displays symptoms. 
Then for a random root node $u^*$,
\begin{align*}
 \flood_1(G)
 &\weqd \max_v W_G(u^*,v), \\
 \flood_2(G)
 &\weqd \max_v \Big( W_G(u^*,v) + X_1(v) \Big).
\end{align*}
\cred{
Applying formula \eqref{eq:DelayedFloodingTypical} in Theorem~\ref{the:DelayedFloodingTime} with $\lambda_1 = \infty$ corresponding to $X_{1}(v) = 0$, we have \whp,
\begin{align}
\label{eq:flood1}
 \flood_1(G)
 &\weq \left(\frac{1}{\lambda(\nu-1)} + \frac{1}{\lambda\delta}\right) \log n + o(\log n).
\end{align}
Note that the same formula can also be obtained from \eqref{eq:UndelayedFloodingTypical}. Applying \eqref{eq:DelayedFloodingTypical} again, we have \whp,
\begin{align}
 \label{eq:flood2}
 \flood_2(G)
 &\weq \left(\frac{1}{\lambda(\nu-1)} + \frac{1}{\lambda\delta\wedge\lambda_1}\right) \log n + o(\log n).
\end{align}
These two formulas lead to the following results.
}
\begin{corollary}
\label{the:Broadcasting}
For a random $\delta$-regular graph $G$ on $n$ nodes with $\delta \ge 3$, when the distribution of $X_1(v)$ has an exponential tail of rate $\lambda_1$ according to \eqref{eq:ExpTail},
\[
 \flood_1(G)
 \weq \frac{2}{\kappa} \left( \frac{\delta-1}{\delta-2} \right) \log n + o(\log n)
\]
and
\[
 \flood_2(G)
 \weq \left( \frac{\delta}{\kappa(\delta-2)} + \frac{1}{\kappa \wedge \lambda_1} \right) \log n + o(\log n)
\]
with high probability as $n \to \infty$.
\end{corollary}
\cred{
\begin{proof}
The results follow directly by substituting $\lambda = \delta/ \kappa$ and $\nu = \delta -1$ into \eqref{eq:flood1} and \eqref{eq:flood2}. The coefficient in \eqref{eq:flood1} simplifies by direct calculation into
\begin{align*}
 \frac{1}{\lambda(\nu-1)} + \frac{1}{\lambda\delta} = \frac{\delta}{\kappa(\delta-2)}+ \frac{1}{\kappa}  = \frac{2}{\kappa}\left(\frac{\delta-1}{\delta-2}\right).
\end{align*}
\end{proof}
}
\begin{figure}[h]
\centering
\includegraphics{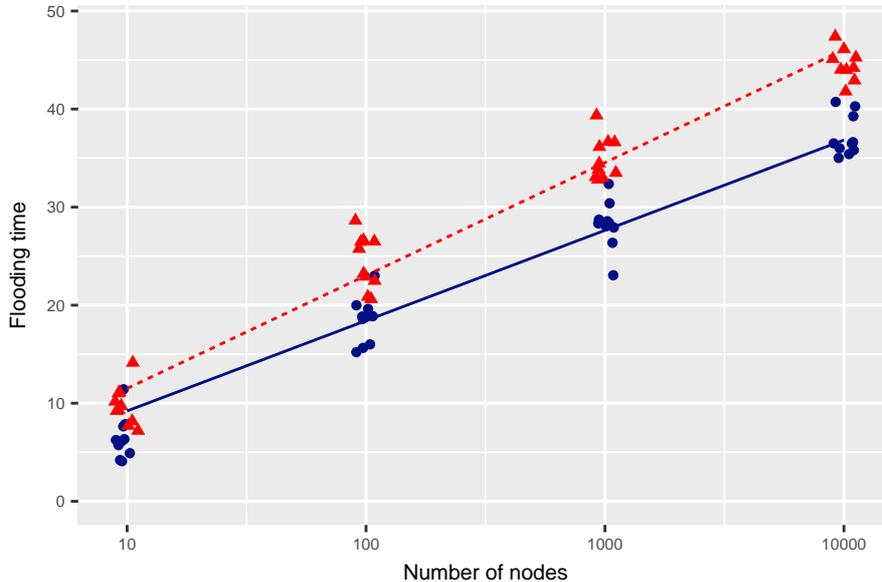}
\caption{\label{fig:FloodingTimes} Flooding times on random $\delta$-regular graphs with $\delta = 3$, $\kappa = 1$, and $\lambda_1=1/2$. Blue circles represent simulated values of $\flood_1(G)$. Red triangles represent simulated values of $\flood_2(G)$. The blue solid line and the red dashed line correspond to the limiting formulas of Corollary~\ref{the:Broadcasting}. (Color online.)} 
\end{figure}

Figure~\ref{fig:FloodingTimes} illustrates how the limiting approximations of Corollary~\ref{the:Broadcasting} relate to simulated values of the flooding times on 3-regular graphs. The sizes of the fluctuations around the theoretical values appear to be of constant order with respect to $n$. A constant order of fluctuations corresponds to the well-known fact in statistical extreme value theory that the maximum of $n$ independent exponential random numbers is approximately Gumbel-distributed around a value of size $\log n$. However, the additional randomness induced by the underlying random graph may cause the fluctuations to grow slowly with respect to $n$. Whether or not the fluctuations grow with $n$ is not possible to detect from simulations of modest size, because the growth rate of the fluctuations is at most $o(\log n)$.

Figure~\ref{fig:FloodingDynamics} describes simulated trajectories of node counts in different states in a random 3-regular graph of 1000 nodes. The trajectories are approximately S-shaped, with random horizontal shifts caused by the initial and final phases of the process.

\begin{figure}[h]
\centering
\includegraphics{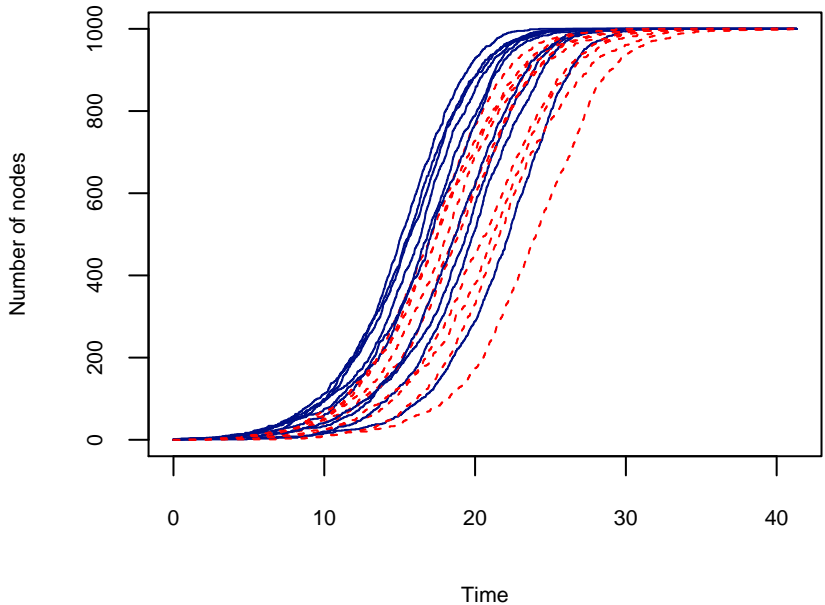}
\caption{\label{fig:FloodingDynamics} Simulated trajectories of the number of nodes in state~1 or state~2 (blue, solid) and the number of nodes in state 2 (red, dashed) for a random $\delta$-regular graph with $n=1000$, $\delta = 3$, $\kappa = 1$, and $\lambda_1=1/2$. (Color online.)} 
\end{figure}

\section{Proofs}
\label{Proof of the main result}

\subsection{Configuration model}

A standard method for studying the random graph $G = G(n, d^{(n)})$ is to investigate a related random multigraph. A \emph{multigraph} is a triplet $G = (V, E, \phi)$, where $V$ and $E$ are finite sets and $\phi: E \to \binom{V}{1} \cup \binom{V}{2}$. Here $\phi(e)$ refers to the set of one (loop) or two (non-loop) nodes incident to $e \in E$. A multigraph is called \emph{simple} if $\phi$ is one-to-one (no parallel links) and $\phi(E) \subset \binom{V}{2}$ (no loops). The degree of a node $i$ is defined by $\sum_{e \in E} \cred{\big(} 1( i \in \phi(e)) + 1( \{i\} = \phi(e) )\cred{\big)}$, that is, the number of links incident to $i$, with loops counted twice. A path of length $k \ge 0$ from $x_0$ to $x_k$ is a set of distinct nodes $\{x_0,x_1, \dots, x_k\}$ such that $\{x_{j-1}, x_j\} \in \phi(E)$ for all $j$. For a multigraph $G$ weighted by $W: E \to (0,\infty)$, we denote
\[
 W_G(u,v) \weq \inf_{\Gamma} \, \sum_{e \in \Gamma} W(e),
\]
where $\Gamma$ is the set of paths from $u$ to $v$. When $G$ is connected, the above formula defines a metric on $G$. 

\cred{
Let us recall the usual definition of the configuration model in \cite{Bollobas_1980}. Let $n$ be a positive integer and $d = (d_1,d_2,...,d_n)$ be a sequence of nonnegative integers. For each node $i\in [n]$ we attach $d_i$ distinct elements called half-edges. A pair of half-edges is called an edge. To obtain a random multigraph $G^*$, it is required that the sum of half-edges $d = (d_1,d_2,...,d_n)$ is even $\sum_{i=1}^{n}d_i = 2m$, where $m$ refers to the number of edges. Let $D_i$ be the set of half-edges of node $i$. Then the size of the set $D_i$ is $d_i$ and the sets $D_1,D_2,...,D_n$ are disjoint. Let $D= \bigcup_{i= 1}^{n}D_i$ be the collection of all the half-edges and let $E$ be a pairing of $D$ (partition into $m$ pairs) selected uniformly at random. The configuration model $G^* = G^*(n, d)$ is the multigraph $([n],E, \phi)$, where the function $\phi:E\to \binom{n}{1}\cup\binom{n}{2}$ is defined by $\phi(e) = \{i \in [n]: D_i \cap e \ne \emptyset\}$.
}
A key feature of the configuration model is that the conditional distribution of $G^*(n,d)$ given that $G^*(n,d)$ is simple equals the distribution of the random graph $G(n, d)$. Moreover, for a sequence of degree lists $d^{(n)}$ satisfying the regularity conditions \eqref{eq:DegreeLim}--\eqref{eq:DegreeVariance}, the probability that $G^*(n, d^{(n)})$ is simple is bounded away from zero \cite{Janson_Luczak_2009}. Therefore, any statement concerning $G^*(n, d^{(n)})$ which holds with high probability, also holds for $G(n, d^{(n)})$ with high probability. This is why in the sequel, we write $G$ in place of $G^*$ and the analysis of weighted distances will be conducted on the configuration model.

\subsection{Notations}

For a node $u$ in the weighted multigraph, we denote by $B(u,t) = \{W_G(u,v) \le t\}$ the set of nodes within distance $t \in [0,\infty]$ from $u$. For an integer $k \ge 0$, we define
\[
 T_{u}(k)
 \weq \min\{t \ge 0: \abs{B(u,t)} \ge k+1\},
\]
with the convention that $\min \emptyset = \infty$. We also denote by $S_{u}(k)$ the number of outgoing links from set $B(u,T_{u}(k))$. Then for any $k$ less than the component size of $u$:
\begin{itemize}
\item $T_{u}(k)$ equals the distance from $u$ to its $k$-th nearest neighbor, and
\item $S_u(k)$ equals the number of outgoing links from the set of nodes consisting of $u$ and its $k$ nearest neighbors.
\end{itemize}
Moreover, $T_u(k) = \infty$ and $S_u(k) = 0$ for all $k$ greater or equal to the component size of $u$.

Throughout in the sequel, we assume that $G$ satisfies the regularity conditions \eqref{eq:DegreeLim}--\eqref{eq:DegreeMinMax}. We introduce the scale parameters
\begin{align*}
 \alpha_n &= \lfloor \log^3 n \rfloor,\\
 \beta_n &=\Big\lfloor 3\sqrt{\tfrac{\mu}{\nu-1}n\log n} \Big\rfloor
\end{align*}
and, with high probability, \cite[Proposition 4.2]{Amini_Draief_Lelarge_2013} (see alternatively \cite[Lemma 3.3]{Ding_Kim_Lubetzky_Peres_2010} or \cite[Proposition 4.9]{Bhamidi_VanDerHofstad_Hooghiemstra_2010})
\begin{equation} 
 \label{eq:DistanceUpper}
 W_G(u,v)
 \wle T_{u}(\beta_n) + T_v(\beta_n)
\end{equation}
for all nodes $u$ and $v$ in the graph $G$. We will next analyze the behavior of $T_{u}(\beta_n)$ and $T_v(\beta_n)$ in typical (uniformly randomly chosen node) and extremal cases.

\subsection{Upper bound on weighted distances}

The following upper bound on the weighted distances is a sharpened version of \cite[Lemmas 4.7, 4.12]{Amini_Draief_Lelarge_2013}. Below we assume that $X \ge 0$ is an arbitrary random number and $u^*$ is a uniformly randomly chosen node, such that $X$, $u^*$, and the graph $G$ are mutually independent, and independent of the weights $(W(e))_{e \in E(G)}$, where weights $W(e)$ are exponentially distributed with rate $\lambda > 0$. We use $\cF_{S_{u^*}}$ to denote the sigma-algebra generated by $S_{u^*} = (S_{u^*}(0), \dots, S_{u^*}(n-1))$.  

\begin{lemma}
\label{the:UpperBoundDistancesGen}
Fix integers $0 \le a < b < n$ and numbers $c_1, c_2 \ge 0$, and let $\cR$ be an $\cF_{S_{u^*}}$-measurable event on which $S_{u^*}(k) \ge c_1 + c_2 k$ for all $a \le k \le b-1$. \cred{For any $0 < \theta <  \lambda (c_1+c_2 a)$},
\begin{multline}
 \label{eq:UpperBoundDistancesZero}
 \cred{
 \E(e^{\theta (T_{u^*}(b) - T_{u^*}(a) + X)}\mid \cR)
 }\\
 \wle M_X(\theta) \exp \left( \frac{\theta}{\theta_1 - \theta}
 + \frac{\theta}{\theta_0} \left( \frac{1}{a+1} + \log \frac{b-1}{a+1} \right) \right),
\end{multline}
where $M_X(\theta) = \E e^{\theta X}$, $\theta_0 = \lambda c_2 - \frac{(\theta - \lambda c_1)_+}{a+1}$ and $\theta_1 = \lambda (c_1+c_2 a)$.
\end{lemma}
\begin{proof}
A key property of the model is that conditionally on $S_{u^*}$, the random numbers  $T_{u^*}(k+1) - T_{u^*}(k)$
are independent and exponentially distributed with rates $\lambda S_{u^*}(k)$.  On the event $\cR$, we see that $\lambda S_{u^*}(a) \ge \theta_1$, and for all $a+1 \le k \le b-1$,
\begin{align*}
 \lambda S_{u^*}(k) - \theta \wge \lambda c_1 + \lambda c_2 k - \theta  \wge \left( \lambda c_2 - \frac{(\theta - \lambda c_1)_+}{k} \right) k \wge \theta_0 k.
\end{align*}
\cred{As a consequence,}
\begin{align*}
 \E(e^{\theta (T_{u^*}(b) - T_{u^*}(a) + X)}\mid \cred{\cF_{S_{u^*}}})  
 &\weq M_X(\theta) \prod_{k=a}^{b-1} \frac{\lambda S_{u^*}(k)}{\lambda S_{u^*}(k) - \theta} \\
 &\weq M_X(\theta)
 \prod_{k=a}^{b-1} \left( 1 + \frac{\theta}{\lambda S_{u^*}(k) - \theta} \right) \\
 &\wle M_X(\theta)
 \exp \left( \sum_{k=a}^{b-1} \frac{\theta}{\lambda S_{u^*}(k) - \theta} \right).
\end{align*}
\cred{Separating the first term from the sum and by the choice of the event $\cR$, we obtain
\begin{align*}
 \E(e^{\theta (T_{u^*}(b) - T_{u^*}(a) + X)}\mid \cR)
 &\wle M_X(\theta)
 \exp \left( \frac{\theta}{\theta_1 - \theta} + \frac{\theta}{\theta_0} \sum_{k=a+1}^{b-1} \frac{1}{k} \right).
\end{align*}
By integration we have $\sum_{k = m}^{n}\frac{1}{k}\le \log(\frac{n}{m-1})$ for any integers $2\le m<n$. Hence, separating the first term again from the sum, we have the desired result,
\begin{align*}
 \E(e^{\theta (T_{u^*}(b) - T_{u^*}(a) + X)}\mid \cR)
 &\wle M_X(\theta)
 \exp \left( \frac{\theta}{\theta_1 - \theta} + \frac{\theta}{\theta_0} \left( \frac{1}{a+1} + \log \frac{b-1}{a+1} \right) \right).
\end{align*}
}
\end{proof}

\subsection{Upper bounds on nearest neighbor distances}

\begin{proposition}
\label{the:InitialGrowth}
For any $0 \le p \le 1$ and $\epsilon > 0$, any random variable $X \ge 0$ independent of $G$,
\begin{align*}
 \pr\left( T_{u^*}(\alpha_n) + X > \Big( \frac{p}{\lambda \delta \wedge \theta^*} + \epsilon \Big) \log n \right)
 \weq o(n^{-p}),
\end{align*}
where $\theta^* = \sup\{ \theta \ge 0: \E e^{\theta X} < \infty\} > 0$.
\end{proposition}
\begin{proof}
Let $t_n = (\frac{p}{\lambda \delta \wedge \theta^*} + \epsilon ) \log n$. An upper bound for the event under study $\cA_n = \{T_{u^*}(\alpha_n) + X > t_n\}$ is obtained by
\begin{equation}
\label{eq:UpperSplit}
\begin{aligned}
 \pr\left( \cA_n \right)
 &\wle \cred{\pr(\cA_n\mid\cR_1)} + \cred{\pr(\cA_n\mid\cR_2\cap\cR_1^c)} \, \pr(\cR_1^c) + \pr(\cR_2^c),
\end{aligned}
\end{equation}
where 
\begin{align*}
 \cR_1 &\weq \left\{ S_{u^*}(k) \ge \delta + (\delta-2)k \quad \text{for all $0 \le k \le \alpha_n -1$} \right\}, \\
 \cR_2 &\weq \left\{ S_{u^*}(k) \ge 1 + (\delta-2)k \quad \text{for all $0 \le k \le \alpha_n -1$} \right\}.
\end{align*}
\if
fa
\cred{
In these both events we require that, whenever a new nearest neighbor of $u^*$ is found, at least $\delta-2$ new outgoing half-edges are reached. Event $\cR_1$ holds when each outgoing half-edge leads to a new nearest neighbor of $u^*$ until the last $(\alpha_n-1)$ new nearest neighbor is found. While, event $\cR_2$ occurs when at most $\delta-1$ half-edges does not lead to a new nearest neighbor of $u^*$ for each step $0\le k \le \alpha_n -1$.
}
\fi
We will next analyze the conditional probabilities in \eqref{eq:UpperSplit}.

(i) To obtain an upper bound of $\cred{\pr(\cA_n\mid\cR_1)}$, by applying Lemma~\ref{the:UpperBoundDistancesGen} with $a=0$, $b = \alpha_n$, $c_1= \delta$, and $c_2 = \delta-2$ and  Markov's inequality, we find that
\[
 \cred{\pr(\cA_n\mid\cR_1)}
 \wle M_X(\theta) \exp \left( \frac{\theta}{\theta_1 - \theta}
 + \frac{\theta}{\theta_0} \left( 1 + \log \alpha_n \right) - \theta t_n \right) \\
\]
for all $0 < \theta < \theta_1 \wedge \theta^*$\cred{
, where $\theta_0 = \lambda (\delta-2)$ and $\theta_1 =\lambda\delta$
}. Now we may choose 
\cred{
$\theta\ge (1-\frac{\lambda\delta\wedge\theta^{*}}{2(p+\epsilon(\lambda\delta\wedge\theta^{*}))}\epsilon)(\lambda\delta\wedge\theta^{*})$
to have} $\theta t_n \ge (p + \frac12 \epsilon (\theta_1 \wedge \theta^*))\log n$. 
\cred{Note that $\theta$ can be arbitrary close to its maximum value $\lambda\delta\wedge\theta^{*}$ if we choose $\epsilon>0$ to be sufficiently small.
Since the constant term $\frac{\theta}{\theta_1-\theta}$ is negligibly small compared to $\log\alpha_n = \Theta(\log\log n)$, we have for large values of $n$, 
}
\[
 \frac{\theta}{\theta_1 - \theta}
 + \frac{\theta}{\theta_0} \left( 1 + \log \alpha_n \right)
 \wle \frac14 \epsilon (\theta_1 \wedge \theta^*) \log n.
\]
\cred{
These two inequalities imply that
}
\begin{equation}
\label{eq:UpperSplit1}
 \cred{\pr(\cA_n\mid\cR_1)}
 \wle M_X(\theta) n^{-(p+\frac14 \epsilon (\lambda \delta \wedge \theta^*))}
 \weq o(n^{-p}).
\end{equation}

(ii) For an upper bound of $\pr_{\cR_2 \setminus \cR_1}( \cA_n)$,
we apply Lemma~\ref{the:UpperBoundDistancesGen} with $a=0$, $b = \alpha_n$, $c_1= 1$, and $c_2 = \delta-2$ and Markov's inequality to conclude that
\[
 \cred{\pr(\cA_n\mid\cR_2\cap\cR_1^c)}
 \wle M_X(\theta) \exp \left( \frac{\theta}{\lambda - \theta}
 + \frac{\theta}{\lambda (\delta-2)} \left( 1 + \log \alpha_n \right) - \theta t_n \right)
\]
for all $0 < \theta < \lambda \wedge \theta^*$. For any such $\theta$, we see that
$
 \theta t_n \ge \epsilon_1 \log n
$
with $\epsilon_1 = \theta (\frac{p}{\lambda \delta \wedge \theta^*} + \epsilon ) > 0$.
Because
\[
 \frac{\theta}{\lambda - \theta}
 + \frac{\theta}{\lambda (\delta-2)} \left( 1 + \log \alpha_n \right)
 \wle \frac12 \epsilon_1 \log n
\]
for all large $n$, it follows that
\begin{equation}
\label{eq:UpperSplit2}
 \cred{\pr(\cA_n\mid\cR_2\cap\cR_1^c)}
 \weq O(n^{- \epsilon_1/2}).
\end{equation}

\cred{
Note that our $S_u(k)$ has the same distribution as the exploration process in \text{Section $4.1$} in \cite{Amini_Draief_Lelarge_2013}. Hence}
 by \cite[Lemma 4.6]{Amini_Draief_Lelarge_2013}, $\pr\cred{(\cR_1^c)} = o(n^{-1} \log^{10} n)$ and $\pr\cred{(\cR_2^c)} = o(n^{-3/2})$. Hence by substituting the bounds \eqref{eq:UpperSplit1} and \eqref{eq:UpperSplit2} into \eqref{eq:UpperSplit} it follows that
\[
 \pr\left( \cA_n \right)
 \wle o(n^{-p}) + O(n^{- \epsilon_1/2})  o(n^{-1} \log^{10} n) + o(n^{-3/2})
 \weq o(n^{-p}).
\]
\end{proof}

\subsection{Upper bounds on moderate distances}

\begin{proposition}
\label{the:ModerateGrowth}
For any $\epsilon > 0$,
\begin{align*}
 \pr\left(T_{u^*}(\beta_n) - T_{u^*}(\alpha_n) > (\frac{1}{2\lambda(\nu-1)}+\epsilon)\log n \right)
 \weq o(n^{-1}).
\end{align*}
\end{proposition}
\begin{proof}
Denote $c = \frac{1}{2\lambda(\nu-1)}$ and  $t_n = (c + \epsilon) \log n$. Set $c_1=0$ and $c_2 = 1/\lambda c$. Fix a number $\theta > \frac{2}{\epsilon}$, and set $\theta_0 = \lambda c_2 - \frac{\theta}{\alpha_n+1}$ and $\theta_1 = \lambda c_2 \alpha_n$. Then for all sufficiently large values of $n$, we see that $0 < \theta < \theta_1$. When we apply Lemma~\ref{the:UpperBoundDistancesGen} with $X=0$, $a = \alpha_n$, and $b = \beta_n$ and Markov's inequality, we find that on the event $\cR_3$ that $S_{u^*}(k) \ge c_2 k$ for all $\alpha_n \le k \le \beta_n -1$,
\begin{align*}
 &\cred{\pr(T_{u^*}(\beta_n) - T_{u^*}(\alpha_n) > t_n\mid \cR_3)} \\
 &\wle \exp \left( \frac{\theta}{\theta_1 - \theta}
 + \frac{\theta}{\theta_0} \left( \frac{1}{\alpha_n} + \log \frac{\beta_n}{\alpha_n} \right) - \theta t_n \right) \\
 &\weq \exp \left( \frac{\theta}{\lambda c_2 \alpha_n - \theta}
 + \frac{\theta}{\lambda c_2 - \frac{\theta}{\alpha_n+1}} \left( \frac{1}{\alpha_n} + \log \frac{\beta_n}{\alpha_n} \right)
 - (c + \epsilon)\theta \log n \right).
\end{align*}
Note that $\beta_n/\alpha_n \wle n$. Because $\alpha_n \to \infty$, we see that
\begin{align*}
\cred{\pr(T_{u^*}(\beta_n) - T_{u^*}(\alpha_n) > t_n\mid \cR_3)} &\wle \exp \Big( 
 (c + \epsilon/2) \theta \log n
 - (c + \epsilon)\theta \log n\Big)\\ 
 &\weq \exp \left(-\frac{\epsilon}{2}\theta \log n \right).
\end{align*}
Due to our choice of $\theta$, the right side is $o(n^{-1})$. The claim follows from this because $\pr(\cR_3^c) =  o(n^{-3/2})$ by \cite[Lemma 4.9]{Amini_Draief_Lelarge_2013}.
\end{proof}

\subsection{Proof Theorem~\ref{the:DelayedFloodingTime}: Upper bounds}
\label{Proof of the upper bound}

Observe that $\E e^{\theta X_i(v)}$ is finite for $\theta < \lambda_i$ and infinite for $\theta > \lambda_i$ due to our assumption on exponential tails \eqref{eq:ExpTail}.  Hence by applying Proposition~\ref{the:InitialGrowth} with $p=1$,
\begin{align*}
 \pr\left( T_{v^*}(\alpha_n) + X_i(v^*) > \left( \frac{1}{\lambda \delta \wedge \lambda_i} +\epsilon \right) \log n \right)
 \weq o(n^{-1}),
\end{align*}
so that by applying the generic union bound
\begin{equation}
 \label{eq:GenericUnion}
 \pr( \max_v X(v) > t )
 \wle \sum_v \pr(X(v) > t)
 \weq n \pr(X(v^*) > t)
\end{equation}
it follows that
\begin{equation}
 \label{eq:UpperFirst}
 \max_v (T_v(\alpha_n) + X_i(v))
 \wle \left( \frac{1}{\lambda \delta \wedge \lambda_i} +\epsilon \right) \log n
 \qquad \whp
\end{equation}
Furthermore, by applying Proposition~\ref{the:InitialGrowth} with $p=0$, it follows that
\begin{equation}
 \label{eq:UpperSecond}
 T_{v^*}(\alpha_n)
 \wle T_{v^*}(\alpha_n) + X_i(v^*)
 \wle \epsilon \log n
 \qquad \whp,
\end{equation}
and by Proposition~\ref{the:ModerateGrowth} and the generic union bound \eqref{eq:GenericUnion}, \whp,
\begin{equation}
 \label{eq:UpperThird}
 T_{v^*}(\beta_n) - T_{v^*}(\alpha_n)
 \wle \max_v \Big( T_v(\beta_n) - T_v(\alpha_n) \Big)
 \wle \left( \frac{1}{2\lambda(\nu-1)} + \epsilon \right) \log n
\end{equation}
By combining \eqref{eq:UpperFirst} and \eqref{eq:UpperSecond} with \eqref{eq:UpperThird}, we conclude that \whp,
\begin{equation}
 \label{eq:MaxBetaUB}
 \max_v (T_v(\beta_n) + X_i(v))
 \wle \left( \frac{1}{\lambda \delta \wedge \lambda_i} + \frac{1}{2\lambda(\nu-1)} + 2 \epsilon \right) \log n
\end{equation}
and
\begin{equation}
 \label{eq:AvgBetaUB}
 T_{v^*}(\beta_n)
 \wle \left( \frac{1}{2\lambda(\nu-1)} + 2\epsilon \right) \log n.
\end{equation}

To prove an upper bound for \eqref{eq:DelayedPassageTypical}, observe that the distribution of $X_i(v^*)$ does not depend on the scale parameter $n$. Therefore, $X_i(v^*) \le \epsilon \log n$ with high probability. In light of \eqref{eq:DistanceUpper} and \eqref{eq:AvgBetaUB}, it follows that, \whp,
\begin{align*}
 W(u^*, v^*)
 &\weq X_0(u^*) + W_G(u^*, v^*) + X_1(v^*) \\
 &\wle X_0(u^*) + T_{u^*}(\beta_n) + T_{v^*}(\beta_n) + X_1(v^*) \\
 &\wle \left( \frac{1}{\lambda(\nu-1)} \log n + 6 \epsilon \right) \log n.
\end{align*}

To prove an upper bound for \eqref{eq:DelayedFloodingTypical}, observe that by applying \eqref{eq:DistanceUpper}, \eqref{eq:MaxBetaUB} and \eqref{eq:AvgBetaUB}, with high probability,
\begin{align*}
 \max_{v} W(u^*, v)
 &\weq \max_{v} \Big( X_0(u^*) + W_G(u^*, v) + X_1(v) \Big) \\
 &\wle \max_{v} \Big( X_0(u^*) + T_{u^*}(\beta_n) + T_v(\beta_n) + X_1(v) \Big) \\
 &\weq X_0(u^*) + T_{u^*}(\beta_n) + \max_{v} \Big( T_v(\beta_n + X_1(v)) \Big) \\
 &\wle \left( \frac{1}{2\lambda(\nu-1)} + 3 \epsilon \right) \log n
  + \left( \frac{1}{\lambda \delta \wedge \lambda_1} + \frac{1}{2\lambda(\nu-1)} + 2 \epsilon \right) \log n \\
 &\weq \left( \frac{1}{\lambda \delta \wedge \lambda_1} + \frac{1}{\lambda(\nu-1)} + 5 \epsilon \right) \log n.
\end{align*}

Finally, for an upper bound for \eqref{eq:DelayedFloodingMax}, observe that by \eqref{eq:DistanceUpper}, with high probability,
\begin{align*}
 \max_{u,v} W(u, v)
 &\weq \max_{u,v} \Big( X_0(u) + W_G(u, v) + X_1(v) \Big) \\
 &\wle \max_{u,v} \Big( X_0(u) + T_{u}(\beta_n) + T_v(\beta_n) + X_1(v) \Big) \\
 &\weq \max_u \Big( X_0(u) + T_u(\beta_n) \Big) + \max_v \Big( X_1(v) + T_v(\beta_n) \Big).
\end{align*}
And hence by \eqref{eq:MaxBetaUB}, it follows that, with high probability,
\begin{align*}
 \max_{u,v} W(u, v)
 &\wle \left( \frac{1}{\lambda \delta \wedge \lambda_0} + \frac{1}{\lambda (\nu-1)} + \frac{1}{\lambda \delta \wedge \lambda_1}  + 4 \epsilon \right) \log n.
\end{align*}
The above inequalities are sufficient to confirm the upper bounds in Theorem~\ref{the:DelayedFloodingTime} because $\epsilon > 0$ can be chosen arbitrarily small.

\subsection{Proof Theorem~\ref{the:DelayedFloodingTime}: Lower bounds}
\label{Proof of the lower bound}
The lower bounds are relatively straightforward generalizations of analogous results \eqref{eq:UndelayedPassageTypical}--\eqref{eq:UndelayedFloodingMax} for the model without node weights, which imply that for an arbitrarily small $\epsilon > 0$, the weighted graph distance $W_G$ satisfies \whp,
\begin{align}
 \label{eq:PassageTypical}
 \frac{W_G(u^*,v^*)}{\log n}
 &\wge \frac{1}{\lambda(\nu-1)} - \epsilon \\
 \label{eq:FloodingTypical}
 \frac{\max_v W_G(u^*,v)}{\log n}
 &\wge \frac{1}{\lambda(\nu-1)} + \frac{1}{\lambda \delta} - \epsilon \\
 \label{eq:FloodingMax}
 \frac{\max_{u,v} W_G(u,v)}{\log n}
 &\wge \frac{1}{\lambda(\nu-1)} + \frac{2}{\lambda \delta} - \epsilon. 
\end{align}
\cred{
We first prove the following lemma and we apply it later in the proof of the lower bounds.
}
\begin{lemma}
\label{the:MaxSumLB}
For every integer $n \ge 1$, let $A_n(i)$ and $B_n(i)$ be random numbers indexed by a finite set $i \in I_n$. Assume that
$(A_n(i))_{i \in I_n}$ are independent and identically distributed, and
\[
 \abs{I_n} \, \pr(A_n(i) > a_n)
 \ \to \ \infty,
\]
and that $B_n(i^*) > b_n$ with high probability, where $i^*$ is a uniformly random point of $I_n$, independent of $(B_n(i))_{i \in I_n}$. Assume also that $A_n(i)$ and $B_n(i)$ are independent for every $i \in I_n$. Then
\[
 \max_{i \in I_n} \left( A_n(i) + B_n(i) \right)
 \ > \ a_n + b_n
\]
with high probability.
\end{lemma}
\begin{proof}
Let
\[
 M_n
 \weq \abs{\{ i \in I_n: A_n(i) > a_n \}}
\]
and
\[
 N_n
 \weq \abs{\{ i \in I_n: A_n(i) > a_n, \, A_n(i) + B_n(i) > a_n + b_n \}}.
\]
Observe that $M_n$ is binomially distributed with $\abs{I_n}$ trials and rate parameter $p_n = \pr(A_n(i) > a_n)$.
Then $\E M_n = \abs{I_n} p_n$ and $\Var(M_n) \le \E M_n$, and because $\abs{I_n} p_n \to \infty$, it follows that $M_n \ge \frac12 \abs{I_n} p_n$ with high probability. Moreover,
\begin{align*}
 \E (M_n - N_n)
 &\weq \sum_{i \in I_n} \pr \left( A_n(i) > a_n, \, A_n(i) + B_n(i) \le a_n + b_n \right) \\
 &\wle \sum_{i \in I_n} \pr \left( A_n(i) > a_n, \, B_n(i) \le b_n \right) \\
 &\weq \sum_{i \in I_n} \pr ( A_n(i) > a_n ) \pr  (B_n(i) \le b_n ) \\
 &\weq \abs{I_n} p_{n} \, \pr( B_n(i^*) \le b_n ).
\end{align*}
Because $\pr( B_n(i^*) \le b_n ) = o(1)$, Markov's inequality implies that $M_n - N_n \le \frac14 \abs{I_n} p_n$ with high probability. We conclude that, with high probability,
\[
 N_n
 \weq M_n - (M_n-N_n)
 \wge \frac12 \abs{I_n} p_n - \frac14 \abs{I_n} p_n
 \weq \frac14 \abs{I_n} p_n
 \wge 1,
\] 
and $\max_{i \in I_n} (A_n(i) + B_n(i)) > a_n + b_n$.
\end{proof}

(i) A suitable lower bound for \eqref{eq:DelayedPassageTypical} follows immediately from \eqref{eq:PassageTypical} because $W(u^*, v^*) \ge W_G(u^*, v^*)$ almost surely.

(ii) To prove a lower bound for \eqref{eq:DelayedFloodingTypical}, note that the exponential tail assumption \eqref{eq:ExpTail} implies that
\[
 n \, \pr \left( X_1(v) > \big( \frac{1}{\lambda_1} - \epsilon \big) \log n \right)
 \ \to \ \infty.
\]
Then by applying Lemma~\ref{the:MaxSumLB} (with $A_n(v) = X_1(v)$, $B_n(v) = W_G(u^*,v)$, and $I_n = [n]$), recalling \eqref{eq:PassageTypical}, we find that, \whp,
\begin{equation}
\label{eq:FloodingTypicalLB}
\begin{aligned}
 \max_v W(u^*,v)
 &\weq \max_v \Big( X_0(u^*) + W_G(u^*, v) + X_1(v) \Big) \\
 &\wge \max_v \Big( W_G(u^*, v) + X_1(v) \Big) \\
 &\wge \left( \frac{1}{\lambda(\nu-1)} + \frac{1}{\lambda_1} - 2\epsilon \right) \log n.
\end{aligned}
\end{equation}
By noting that $\max_v W(u^*,v) \ge \max_v W_G(u^*, v)$ and applying \eqref{eq:FloodingTypical}, we also obtain
\[
 \max_v W(u^*,v)
 \wge \left( \frac{1}{\lambda(\nu-1)} + \frac{1}{\lambda \delta} - 2\epsilon \right) \log n,
\]
and hence \whp,
\[
 \max_v W(u^*,v)
 \wge \left( \frac{1}{\lambda(\nu-1)} + \frac{1}{\lambda \delta \wedge \lambda_1} - 2\epsilon \right) \log n.
\]

(iii) To prove a lower bound for \eqref{eq:DelayedFloodingMax}, note that for any $u \ne v$,
\begin{align*}
 &\pr \left( \frac{X_0(u) + X_1(v)}{\log n} > \frac{1}{\lambda_0} + \frac{1}{\lambda_1} - 2 \epsilon \right) \\
 &\wge \pr \left( \frac{X_0(u)}{\log n} > \frac{1}{\lambda_0} - \epsilon , \ 
 \frac{X_1(u)}{\log n} > \frac{1}{\lambda_1} - \epsilon \right) \\
 &\weq \pr \left( \frac{X_0(u)}{\log n} > \frac{1}{\lambda_0} - \epsilon \right)
 \pr \left( \frac{X_1(u)}{\log n} > \frac{1}{\lambda_1} - \epsilon \right).
\end{align*}
Then the exponential tail assumption \eqref{eq:ExpTail} implies that
\[
 n(n-1) \pr \left( \frac{X_0(u) + X_1(v)}{\log n} > \frac{1}{\lambda_0} + \frac{1}{\lambda_1} - 2\epsilon \right)
 \ \to \ \infty.
\]

Observe next that if $u^*$ and $v^*$ are independent uniformly random elements of $[n]$, and $i^*$ is a uniformly random event in $I_n = \{ (u,v) \in [n]^2: u \ne v\}$, then
\begin{align*}
 \pr( W_G(u^*, v^*) \in F)
 &\weq \frac{1}{n^2} \sum_u \sum_v \pr( W_G(u, v) \in F) \\
 &\weq \frac{1}{n} \pr( W_G(u^*, u^*) \in F) + \frac{n-1}{n} \pr( W_G(i^*) \in F)
\end{align*}
for all measurable sets $F \subset \R$. Hence \eqref{eq:PassageTypical} implies that, \whp,
\[
 W_G(i^*)
 \ > \ \left( \frac{1}{\lambda(\nu-1)} - \epsilon \right) \log n.
\]
Then we may apply Lemma~\ref{the:MaxSumLB} with $A_n(u,v) = X_0(u) + X_1(v)$ and $B_n(u,v) = W_G(u,v)$, to conclude that, \whp,
\[
 \max_{u,v} W(u,v)
 \wge \left( \frac{1}{\lambda_0} + \frac{1}{\lambda(\nu-1)} + \frac{1}{\lambda_1} - 3\epsilon \right) \log n.
\]

We will next apply Lemma~\ref{the:MaxSumLB} again, this time with $I_n = [n]$, and $A_n(v) = X_1(v)$ and $B_n(v) = \max_u (X_0(u) + W_G(u,v))$, recalling \eqref{eq:FloodingTypical},  to conclude that, \whp,
\[
 \max_{u,v} W(u,v)
 \wge \left( \frac{1}{\lambda \delta} + \frac{1}{\lambda(\nu-1)} + \frac{1}{\lambda_1} - 3\epsilon \right) \log n.
\]
By a symmetrical argument, we also find that, \whp,
\[
 \max_{u,v} W(u,v)
 \wge \left( \frac{1}{\lambda_0} + \frac{1}{\lambda(\nu-1)} + \frac{1}{\lambda \delta} - 3\epsilon \right) \log n.
\]
By combining the three above inequalities with \eqref{eq:FloodingMax}, we may conclude that, \whp,
\[
 \max_{u,v} W(u,v)
 \wge \left( \frac{1}{\lambda \delta \wedge \lambda_0}
 + \frac{1}{\lambda(\nu-1)}
 + \frac{1}{\lambda \delta \wedge \lambda_1} - 3\epsilon \right) \log n.
\]

\ifarxiv
\bibliographystyle{alpha}
\fi
\bibliography{lslReferences}
\end{document}